\documentclass[11pt]{amsart}
\usepackage[paperwidth=17cm, paperheight=22.5cm, bottom=2.5cm, right=2.5cm]{geometry}
\usepackage{amssymb,amsmath,amsthm} 
\usepackage[english]{babel}
\usepackage[utf8x]{inputenc}
\usepackage[T1]{fontenc}
\usepackage{enumerate}
\usepackage{mathtools}
\usepackage{graphicx}
\usepackage{subfigure} 
\usepackage{float}
\graphicspath{{Img/}} 
\usepackage[export]{adjustbox}
\usepackage{tikz}
\usetikzlibrary{matrix}
\usepackage{multirow}
\usepackage{bbm}
\usepackage{dsfont}

\newtheorem{theorem}{Theorem}[section]
\newtheorem{proposition}[theorem]{Proposition}
\newtheorem{cor}[theorem]{Corollary}
\newtheorem{lemma}[theorem]{Lemma}
\newtheorem{que}{Question}
\newtheorem*{que*}{Question}

\theoremstyle{definition}
\newtheorem{dfn}[theorem]{Definition}

\theoremstyle{remark}

\newcommand{\conc}{^\smallfrown}	
\newcommand{\rest}{\upharpoonright} 

\def\cA{{\mathcal{A}}} \def\cB{{\mathcal{B}}} \def\cC{{\mathcal{C}}}             \def\cP{{\mathcal{P}}}    \def\cT{{\mathcal{T}}}      

  \def\bC{{\mathbb{C}}}

\def\c{{\mathfrak{c}}}

\def\p{{\mathfrak{p}}}

\newcommand{\w}{\omega}

\title{Sequentially compact separable spaces}
\author{C\'esar Corral, Alan Dow and Paul Szeptycki}
\date{}

\begin{document}

\begin{abstract} We consider the following variation of the Scarborough-Stone problem: Is $X^\kappa$ always countably compact whenever $X$ is separable and sequentially compact?
    
\end{abstract}

\maketitle

\section{Introduction}

In their 1966 paper {\it Products of nearly compact spaces}, \cite{SS} Scarborough and Stone observed that a product of a family of sequentially compact spaces ``..will in general (when there are uncountably many factors)
not be sequentially compact...'' and noted that it ``would be
interesting to know whether every product of sequentially compact spaces must
be countably compact.''  This question has come to be known as the Scarborough-Stone problem. The updated question asks: {\em is it consistent that any family of sequentially compact spaces has product that is countably compact?} as many consistent counter-examples to the problem are now known 
(see \cite{vaughanSS}). It has also been shown (even in Scarborough and Stone's original paper, though not explicitly) that if for each $p\in \omega^*$ there is a sequentially compact space $X_p$ that is not $p$-compact, then the product $\prod\{X_p:p\in \omega^*\}$ is not countably compact. Moreover, if there is any such family of spaces, then one can sew them together to obtain a single sequentially compact space $X$ which is not $p$-compact for any $p\in \omega^*$ and so the power $X^{2^{\mathfrak c}}$ is not countably compact. This conglomerate space $X$ is, however, not separable and we were surprised to observe that the following natural variant of the Scarborough-Stone problem had not been previously considered:

\begin{que}\label{MainQ} If $X$ is a separable and sequentially compact space, must every power $X^\kappa$ be countably compact?
\end{que}

This seems to be an open question and while we have not solved it, we have some partial results clarifying the situation.

As with the general Scarborough-Stone problem, Question \ref{MainQ} is equivalent to the following:
\begin{que} If $X$ is a separable and sequentially compact space, is $X$ $p$-compact for some $p\in \omega^*$?
\end{que}

We have a consistent counterexample in Miller's model: indeed in that model there is a separable sequentially compact space of size $\mathfrak c$ which is not $p$-compact for any $p$. Towards a positive answer we only have a partial consistency result: assuming $MA$, for every separable sequentially compact space $X$, {\em if $|X|\leq \mathfrak c$} then $X$ is $p$-compact for some $p$.

This raises the related question whether the continuum can be a bound on the cardinality of separable sequentially compact spaces.  E.g., if it were consistent with Martin's Axiom that ${\mathfrak c}$ is a bound on the size of separable sequentially compact spaces, then we would have a consistent positive answer to Question \ref{MainQ} (alas, this is not the case).

\begin{que}\label{SecondQ} If $X$ is separable and sequentially compact, is then $|X|\leq {\mathfrak c}$?
\end{que}

In regards to this question, we have the following results: In the Cohen model, yes, the continuum is a bound. However, we also have examples giving a consistent negative answer: 
 if ${\mathfrak s}>\omega_1$ and $2^{\omega_1}>\mathfrak c$, then there is a large separable sequentially compact space (namely $2^{\omega_1}$). In addition we show that there is a separable sequentially compact spaces of size $2^{\mathfrak p}$, and so, under ${\mathfrak{p=c}}$ there is a separable sequentially compact space of size $>{\mathfrak c}$. 

We present all these results in the remainder of the paper and conclude with some open questions. 

Our terminology and notation are standard and we refer the reader to \cite{kunenbook} for the set-theoretic notions and \cite{engelking} for the topological.




\section{$MA+\neg\text{CH}$ and small separable sequentially compact spaces}

In this section we show that a consequence of $MA+\neg\text{CH}$ implies that every sequentially compact space of cardinality $\leq{\mathfrak c}$ is $p$-compact for some $p\in \omega^*$ and so every power of $X$ is countable compact. To clearly state this consequence, we need to recall the notion of a tree $\pi$-base for $\omega^*$.  

For our purposes, a tree $\pi$-base for $\omega^*$ is a family $T\subseteq [\omega]^\omega$ such that
\begin{enumerate}
\item $T$ is a tree under the ordering $\subseteq^*$
\item Levels of $T$ are infinite maximal almost disjoint families in $[\omega]^{\omega}$
\item For every $A\in [\omega]^\omega$ there is $t\in T$ such that $t\subseteq^* A$ (i.e., a $\pi$-base for $\omega^*$)
\end{enumerate}
It was proven by Balcar, Pelant and Simon \cite{BPS} that there exists a tree $\pi$-base for $\omega^*$.
\begin{theorem} $\omega^*$ admits a tree $\pi$-base.\qed
\end{theorem}

Moreover, assuming $MA$, we have that 
\begin{enumerate}
\item[$(*)$] There is no tree $\pi$-base of height $<\mathfrak c$ and every tree $\pi$-base of height $\c$ includes a cofinal branch.
\end{enumerate}

Let us now show the main result of this section:

\begin{theorem} Assuming $(*)$ every sequentially compact space of cardinality $\leq{\mathfrak c}$ is $p$-compact for some $p\in \omega^*$.
\end{theorem}
\begin{proof} Let $X$ be sequentially compact, and $f:\omega\rightarrow X$ any function. Then we have that
$${\mathcal P}_f=\{A\subseteq \omega: f\upharpoonright A\text{ is a convergent sequence}\}$$
is dense in $[\omega]^\omega$ with respect to $\subseteq ^*$. And so if we take any ${\mathcal A}\subseteq {\mathcal P}_f$ that is maximal almost disjoint, then it is also maximal in $[\omega]^{\omega}$.

Now, let us assume in addition that $|X|\leq {\mathfrak c}$ and enumerate $X^\omega$ as $\{f_\alpha:\alpha\leq \mathfrak c\}$ and enumerate $[\omega]^\omega$ as $\{x_\alpha:\alpha<\mathfrak c\}$. 

We recursively construct mad families ${\mathcal A}_\alpha$ for all $\alpha<\mathfrak c$ such that
\begin{enumerate}
\item For all $\beta<\alpha$, $\mathcal A_\alpha$ refines $\mathcal A_\beta$.
\item $\mathcal A_\alpha\subseteq {\mathcal P}_\alpha:=\cP_{f_\alpha}$.
\item For all $a\in \mathcal A_\alpha$ either $a\subset x_\alpha$ or $a\cap x_\alpha=\emptyset$.
\end{enumerate}
Note that if $\mathcal A_\beta$ has been constructed for all $\beta<\alpha$, and there is no such $\mathcal A_\alpha$ then we would have a tree $\pi$-base for $\omega^*$ of height $\alpha<\mathfrak c$ contradicting our assumption $(*)$. And item $(3)$ assures that the construction yields a tree $\pi$-base of height $\mathfrak c$. And so it has a cofinal branch $b$. By (3) this branch generates an ultrafilter $p$ with the property that $p\cap {\mathcal P}_\alpha\not=\emptyset$ for all $\alpha$ and so each $f_\alpha$ has a $p$-limit. \end{proof}

\section{A sequentially compact separable space that is not $p$-compact.}

\begin{theorem}\label{miller} Assume the following variation of NCF: there is $\{p_\alpha:\alpha<{\mathfrak c}\}$ a family of $P$-points each of character $\omega_1$ such that for every ultrafilter $u$ there is a finite-to-one $f$ and an $\alpha<\mathfrak c$ such that $f(u)=p_\alpha$. Then there is a separable sequentially compact space of size $\mathfrak c$ that is not $p$ compact for any ultrafilter $p$. 
\end{theorem}

The hypotheses of this theorem holds in the Miller model (see \cite{BSNCF}). In fact, NCF is witnessed by the stronger property that any two ultrafilters are mapped by finite-to-one maps to one of the $p_\alpha$'s. 

We first note a consequence of the hypothesis of Theorem \ref{miller}
\begin{lemma} Under the assumptions of the theorem, if $X$ is a space that is not $p_\alpha$ compact for any $\alpha<{\mathfrak c}$. Then $X$ is not $p$-compact for any ultrafilter $p$. 
\end{lemma}
\begin{proof} Fix $p$ and fix $\alpha$ and $f:\omega\rightarrow \omega$ finite-to-one such that $f(p)=p_\alpha$. Fix $x_n\in X$ so that $(x_n:n\in \omega)$ has no $p_\alpha$-limit. Let $y_n=x_{f(n)}$. Take any $x\in X$. Then $x$ is not a $p_\alpha$ limit, so $x$ has a neighborhood $U$ such that $B=\{n:x_n\not\in U\}\in p_\alpha$. Therefore $f^{-1}(B)\in p$ and so $\{n:y_n\not\in U\}\in p$ and so $x$ is not a $p$-limit of the sequence $(y_n:n\in\omega)$.
\end{proof}

To give the construction of the example establishing Theorem \ref{miller}, we need to recall the example of the Franklin space that is sequentially compact but not $p$-compact for a fixed $P$-point $p$ of character $\omega_1$.  Let $\{u_\alpha:\alpha<\omega_1\}$ be a $\subseteq^*$-decreasing sequence in $[\omega]^\omega$ generating $p$ and let 
$X_p$ be a disjoint union of copies of $\omega$ and $\omega_1$ where the points of $\omega$ are isolated and for all $\beta<\alpha\in \omega_1$ and each finite $F\subset \omega$,  
$$U=\{\gamma:\beta<\gamma\leq \alpha\}\cup \left(u_\beta\setminus (u_\alpha\cup F)\right)
$$
is a typical basic open set containing $\alpha$. Clearly, no $\alpha$ is a $p$-limit of $\omega$ so $X_p$ is not $p$-compact.

That this space is sequentially compact follows from the fact that $\omega_1$ as a subspace of $X_p$ has its usual order topology, which is sequentially compact. So we need to check that every infinite sequence in $\omega$ has a convergent subsequence. But for any infinite $A\subseteq \omega$, if $\alpha$ is minimal such that $A\setminus u_\alpha$ is infinite, then $A$ contains an infinite set $B$ such that $B\subseteq^* u_\beta$ for all $\beta<\alpha$ but $B\cap u_\alpha=\emptyset$. I.e., $B$ converges to $\alpha$. 

\vspace{3mm}

{\it Proof of Theorem \ref{miller}. }Let $X=2^{<\omega}\cup (2^\omega \times 2)$ be the double arrow space over the Cantor tree. To be precise, we first fix the lexicographical order on $2^{\leq \omega}$. I.e., for $f\not=g\in 2^{\leq \omega}$, we define $f<g$ if $f\subseteq g$, or if $f(n)<g(n)$ where $n$ is minimal with $f(n)\not=g(n)$ and we extend this order to $X$ by declaring $(f,0)<(f,1)$ for all $f\in 2^\omega$. 
We give $X$ the order topology. 
Let ${\mathcal U}$ be the enumerated family of $P$-points in the given Miller model.  We enumerate $2^\omega$ as $\{f_p:p\in{\mathcal U}\}$ and replace each point $(f_p,0)$ by the Franklin space $X_{p}$ keeping $\{f_p\upharpoonright n:n<\omega\}$ as the copy of $\omega$ in $X_{p}$ and replacing $(f_p,0)$ by $\{f_p\}\times \omega_1\setminus 2$ a copy of $\omega_1$ (we remove $\{0,1\}$ to avoid duplication with the points $(f,1)$ that remain in our space). For each $p\in{\mathcal U}$, let $\{u^p_\alpha:\alpha<\omega_1\}$ be the $\subseteq^*$ decreasing base for $p$.

Of course, we need to clarify the topology on the resulting space which we denote as $Z=2^{<\omega}\cup 2^\omega\times\{1\}\cup 2^\omega\times \omega_1\setminus 2$. 

Points of $2^{<\omega}$ are isolated. For any open interval $I$ in $X$, let 
$$I_Z=\left(I\cap \left(2^{<\omega}\cup 2^\omega\times\{1\}\right)\right)\cup \bigcup_{(f,0)\in I} \{f\}\times (\omega_1\setminus 2)$$
and declare $I_Z$ to be open in $Z$.

Note that for each $f\in 2^\omega$, intervals of the form $I_Z$ where $I=[(f,1),s]$ where $s\in 2^{<\omega}$ such that $f<s$, form a neighborhood base at $(f,1)$ in $Z$.  

To define the neighborhood base at points $(f_p,\alpha)$ for $\alpha<\omega_1$ we need the following notation for $s\in 2^{<\omega}$ and $f\in 2^{\omega}$ such that $s\subseteq f$:
$$[s,f]=\{s\}\cup \{t\in 2^{< \omega} : s\subseteq t\text{ and }t<f\text{ and }t\not\subseteq f\}\cup \{(g,i)\in Z:s\subseteq g<f\}$$

Finally, for $(f_p,\alpha)\in Z$ and $\beta<\alpha$ and $n\in \omega$
Let $$T_{\beta,\alpha, n,f_p}=\bigcup\{[f_p\upharpoonright m,f_p]: m\in (u^p_{\beta}\setminus u^p_\alpha) \setminus n\}$$

The family $\{T_{\beta,\alpha, n,f_p}:n\in \omega, \beta<\alpha\}$ we declare as a neighborhood base at the point $(f_p,\alpha)$. 

It is easy to check that this is a zero-dimensional topology on $Z$.

Note that the subspace $\{f_p\upharpoonright n:n<\omega\}\cup \{f_p\}\times \omega_1$ is a closed copy of the Franklin space $X_p$ (and so let's just denote this subspace as $X_p$).

It immediately follows that $Z$ is not $p$ compact for any $p\in {\mathcal U}$. 

So we need to check that $Z$ is sequentially compact and then the proof of the theorem is complete. 

So fix a sequence $(x_n:n\in \omega)$ and by going to a subsequence we can assume we are in one of the following cases. 

\noindent CASE 1: There is a $p$ such that $x_n\in X_{p}$ for all $n$. In this case by sequential compactness of the Franklin space we have a convergent subsequence. 

\noindent CASE 2: Not case 1. Then by going to a subsequence we have two subcases

Case 2a:  $x_n\in 2^{<\omega}$ for all $n$. In this case we may go to a further subsequence and assume, wlog, that $\{x_n:n\in \omega\}$ is monotone and forms an antichain (otherwise there would a subsequence contained in a branch putting us in CASE 1). Being monotone it must converge in the space $2^{\leq \omega}$ to some $f\in 2^\omega$ in the usual lexicographic order topology. If it is monotone decreasing then the subsequence converges in $X$ to $(f,1)$. But then it also converges to $(f,1)$ in $Z$. If it is monotone increasing, then let $s_n=f \cap x_n$. Then $s_n$ is a subset of the Franklin space $X_p$ where $f=f_p$ and so the $s_n$ converges to some $(f,\alpha)$. But then the sequences $(x_n)$ also converges to $(f,\alpha)$.

Case 2b: Not case 1 or 2a. Then we have a subsequence that is disjoint from $2^{<\omega}$ and we can choose for each $n$ a $g_n$ such that either $x_n=(g_n,i_n)$ for $i_n=1$ for all $n$ or $i_n=\alpha_n$ for some $\alpha_n\in \omega_1$. In any case, by going to a subsequence we can assume that the sequence of $g_n$'s is monotone and converges in the space $2^\omega$ (with respect to the lexicographic order topology) to some $f$. If the subsequence is decreasing then, in $Z$, the sequence of $x_n$'s converge to $(f,1)$ and otherwise if increasing then the same argument as in Case 2a shows that the sequence converges to $(f,\alpha)$ for some $\alpha\in \omega_1$. \qed

\section{The cardinality of sequentially compact separable spaces}

One scenario in which Problem \ref{MainQ} could have a consistent positive answer is to have all separable sequentially compact spaces have cardinality bound by ${\mathfrak c}$ consistent with the principle $(*)$ that every tree $\pi$-base has height $\mathfrak c$ and a cofinal branch.  While this is open, we turn to the weaker question of whether $\mathfrak c$ is a bound on the cardinality of such spaces. As we will see this is consistent and independent. 

A model where $\mathfrak c$ is an upper bound is the Cohen model and this follows from the proof of one of the main results from \cite{DowJuhasz}.
The proof of Theorem 4.2 from that paper actually shows the following:

\begin{theorem}
    
 Let $V$ be a model for {\sf GCH} and $\kappa>\aleph_1$ be a regular cardinal. Then in the forcing extension, $V_{\bC_\kappa}$ we have that for any regular space $X$ if a countable $A\subseteq X$ has closure larger than $\kappa$ then there is $B\subseteq A$ such that $B$ has no nontrivial sequences converging in $X$.\qed
\end{theorem}


And so as an immediate corollary we obtain

\begin{cor}    Let $V$ be a model for {\sf GCH} and $\kappa>\aleph_1$ be a regular cardinal. Then in $V_{\bC_\kappa}$ every sequentially compact separable space has size at most $\c$.\qed
    
\end{cor}

Then it is consistent that separable sequentially compact spaces are bound in size by the continuum.

We present two consistent examples of large (i.e., of cardinality $>{\mathfrak c}$) separable sequentially compact spaces. The first is easily obtained from van Douwen's characterization of sequential compactness of $2^\kappa$ and some known consistency results concerning cardinal invariants of the continuum. Recall the definition of a splitting family and the cardinal $\mathfrak s$:

\begin{dfn} A family ${\mathcal S}\subseteq [\omega]^{\omega}$ is a splitting family if for every $a\in[\omega]^\omega$ there is $s\in {\mathcal S}$ such that $|s\cap a|=|a\setminus s|=\omega$. The minimal cardinality of a splitting family is the splitting number and is denoted ${\mathfrak s}$
\end{dfn}

Van Douwen proved that $\mathfrak s$ is the minimal $\kappa$ such that $2^{\kappa}$ is not sequentially compact \cite{vanDouwentheintegers}. And so, in any model where $\omega_1<\mathfrak s$ and $2^\omega<2^{\omega_1}$, the space $2^{\omega_1}$ would be a separable sequentially compact space of cardinality larger than the continuum. 

To obtain such a model, start with a model of $MA+\mathfrak c=\omega_2$ and so ${\mathfrak s}=\omega_2$ in that model. Then force with the countably closed partial order $Fn(\omega_3,2,<\omega_1\}$ that adds $\omega_3$ Cohen subsets of $\omega_1$. Since no new reals are added, $\mathfrak s=\mathfrak c=\omega_2$ is preserved but $2^{\omega_1}=\omega_3>\mathfrak c$.


Of course $MA$ fails in this model (and indeed, so does the principle $(*)$) and we are more interested in whether large examples can be ruled out by MA or stronger assumptions, so, it is natural to ask whether MA or PFA is consistent with all separable, sequentially compact spaces have size $\leq\mathfrak c$. Alas, the answer is `no'.  

The example will be the Stone space of a so-called {\em $T$-algebra}, a particular type of Boolean algebra. We will show that there is such a Stone-space of size $2^{\frak p}$ and so it follows from MA (even ${\frak p}={\frak c}$) that there is a $T$-algebra whose Stone space is separable, sequentially compact and of cardinality $>\mathfrak c$.

For the remainder of this section we will deal with Boolean algebras and restrict our attention to subalgebras of $\cP(\w)$. Given a Boolean algebra $\cA\subseteq\cP(\w)$ and $x\subseteq\w$, we denote by $\cA\langle x\rangle$ the Boolean algebra generated by $\cA$ and $x$. The Stone space of a Boolean algebra $\cA$ will be denoted by $St(\cA)$ and is the set of all ultrafilters on $\cA$ with the topology generated by the base $\{b^*:b\subseteq\w\}$, where $b^*$ is the family of all ultrafilters containing $b$. 
We will always assume that every Boolean algebra $\cA$ contains all finite sets, hence $St(\cA)$ is always a compact Hausdorff separable space.
A Boolean algebra $\cB$ is minimal over $\cA$ if there is no subalgebra $\cC\subseteq\cB$ that properly contains $\cA$ \cite{koppelberg}. If $\cB$ is a minimal extension over $\cA$, then $\cB=\cA\langle x\rangle$ for some (equivalently for every) $x\in\cB\setminus\cA$. 
It is shown in \cite{koppelberg} that if $\cB$ is a minimal extension of $\cA$ then there is a unique ultrafilter $u$ on $\cA$ which does not generate an ultrafilter on $\cB$.

\begin{dfn}[\cite{koppelberg},\cite{koszmidertalgebras}]
    We say that $x\subseteq\w$ is \emph{minimal} for $(\cA,u)$ if $u$ is the unique ultrafilter of the Boolean algebra $\cA\subseteq\cP(\w)$ which does not generate an ultrafilter in $\cA\langle x\rangle$. 
\end{dfn}

We will make use of the following two results.

\begin{lemma}{\cite{koszmidertalgebras}}\label{Lemmakoszmider}
    Let $\cA\subseteq\cP(\w)$ be a Boolean algebra, and let $u$ be an ultrafilter on $\cA$. Then $x$ is minimal for $(\cA,u)$ if and only if $\w\setminus x$ is minimal for $(\cA,u)$.\qed
\end{lemma}

\begin{proposition}{\cite{koszmidertalgebras}}\label{characterizationminimality}
    Suppose that $\cA\subseteq\cP(\w)$ is a Boolean algebra, $x\in\cP(\w)\setminus\cA$ and $u$ is an ultrafilter on $\cA$. Then the following are equivalent.
    \begin{enumerate}
        \item $x$ is minimal for $(\cA,u)$,
        \item $u=\{a\in \cA:a\cap x\in\cA\}$ and
        \item $u^*=\{a\in \cA:a\cap x\notin\cA\}$.\qed
    \end{enumerate}
\end{proposition}

\vspace{5mm}

A Boolean algebra $\cA$ is minimally generated if it is generated by $\bigcup_{\alpha\in\beta}\cA_\alpha$, where $\cA_{\alpha+1}$ is a minimal extension over $\cA_\alpha$ and for $\alpha$ limit $\cA_\alpha$ is generated by $\bigcup_{\delta<\alpha}A_\delta$. In this paper we will moreover assume that $\cA_0=[\w]^\w$.

\vspace{3mm}

We now describe a powerful tool for building minimally generated Boolean algebras introduced by Koszmider \cite{koszmidertalgebras}. 
Recall that $(T,\leq)$ is tree if it is a partially ordered set where $\{t\in T:t\leq s\}$ is well ordered for every $s\in T$. 
We will work with subtrees of the form $(2^{<\kappa},\subseteq)$. 
A subtree $T\subseteq 2^{<\kappa}$ is an \emph{acceptable tree} if 
\begin{enumerate}
    \item the length of $t$, denoted by $l(t)$, is a successor ordinal for every $t\in T$,
    \item $t\rest(\alpha+1)\in T$ whenever $t\rest(\alpha+1)\subseteq s\in T$ and
    \item $t\conc 0\in T$ if and only if $t\conc 1\in T$ for every $t\in 2^{<\kappa}$.
\end{enumerate}

A branch $b$ in $T$ is a maximal chain in $T$.

\begin{dfn}[\cite{koszmidertalgebras}]
    Given an acceptable tree $T\subseteq2^{<\kappa}$, we say that $\cA$ is a $T$-algebra if $\cA=\langle a_t:t\in T\rangle$ where 
    \begin{enumerate}
        \item each $a_t\in[\w]^\w$,
        \item $a_{t\conc1}=\w\setminus a_{t\conc0}$ for every $t\conc i\in T$ and,
        \item for every $t\in T$, the family $\{a_s:s<t\}$ is centered and $a_t$ is minimal for $(A_t,u_t)$ (where $A_t$ and $u_t$ are the Boolean algebra and the ultrafilter generated by $\{a_s:s<t\}$ respectively).
    \end{enumerate}
\end{dfn}

\vspace{5mm}

We will say that $\cA$ is a $\cT$-algebra if it is a $T$-algebra for some acceptable tree $T$.
$\cT$-algebras are minimally generated and if $\cA$ is a $T$-algebra, then the elements of $St(\cA)$ correspond to maximal chains through $T$ \cite{koszmidertalgebras}. Moreover, if $u\in St(\cA)$ and $b_u$ is the associated branch, then $\{a_t^*:t\in b_u\}$ forms a local base for $u$ in $St(\cA)$. We will keep this notation and use $u_b$ to refer to the ultrafilter $u$ generated by a branch $b$ in a $\cT$-algebra and vice versa, we will use $b_u$ to refer to a branch defining an ultrafilter $u$.\\
To simplify our notation, we slightly modify our definition and assume that $\emptyset\in T$ and $a_\emptyset=\w$ for every acceptable tree $T$.

\vspace{5mm}

We also recall that the \emph{pseudointersection number} $\p$, is defined as the minimal size of a centered family with no infinite pseudointersection.  
We define a labeling $\{a_t:t\in 2^{<\p}\}\subseteq[\w]^\w$ of $2^{<\p}$ for the rest of this section as follows: Let $a_\emptyset=\w$ and assume that for $s\in2^{<\p}$, the family $\{a_t:t<s\}\subseteq[\w]^\w$ has the finite intersection property. Then we can find a pseudointersection $P_s\in[\w]^\w$ for this family. Let $P_s'\in[P_s]^\w$ be such that $|P_s\setminus P'_s|=\w$. 
If $l(s)$ is a limit ordinal define $a_s=P_s$. Otherwise, if $s=t\conc i$, define $a_{t\conc0}=P_s'$ and $a_{t\conc1}=\w\setminus P'_s$.

It is clear from the construction that for any $s\in2^{<\p}$ the family $\{a_t:t<s\}$ has the finite intersection property and that $a_{s\conc i}$ and $a_{s\conc (1-i)}$ are complements.

\begin{proposition}\label{acceptableistalgebra}
    If $T\subseteq 2^{<\p}$ is an acceptable subtree and the $a_t$ are as described above, then $\{a_t:t\in T\}$ generates a $T$-algebra.
\end{proposition}

\begin{proof}
    We only need to check that for $t\in T$, $a_t$ is minimal for $(A_t,u_t)$, where $A_t$ and $u_t$ are as in Definition \ref{Lemmakoszmider}. 
    
    Since the length of any node in $T$ is a successor ordinal, there are $s\in T$ and $i\in2$ such that $t=s\conc i$. By Lemma \ref{Lemmakoszmider}, it suffices to show that $a_{s\conc0}$ is minimal over $(\cA_t,u_t)$.

    Indeed, $a_t=a_{s\conc0}$ is a pseudointersection of $\{a_r:r<t\}$, i.e., $a_t\subseteq^* a_r$ for every $r<t$. That implies $a_r\setminus a_t\in[\w]^{<\w}\subseteq\cA_t$. By Proposition \ref{characterizationminimality}, $a_t$ is minimal for $(\cA_t,u_t)$.
\end{proof}

Before we prove the main result of this section, we will introduce a bit more of notation. If $X$ is a $T$-algebra and $b$ is a branch of $T$, we define $X_b$ as the $T'$-algebra space, where $T'$ is the minimal acceptable subtree of $T$ in which $b$ is a branch (here we are assuming $T$ and $T'$ have the same labeling $a_t$). Hence $X_b=\{z_\alpha:\alpha<l(b)\ \land\ \alpha \textnormal{ is a successor}\}\cup\{u_b\}$, where:
\begin{itemize}
    \item $u_b$ is the ultrafilter generated by $\{a_{b\rest\alpha}:\alpha<l(b)\}$ and
    \item each $z_\alpha$ is the ultrafilter generated by $\{a_{b\rest\beta}:\beta<\alpha\}\cup\{\w\setminus a_{b\rest\alpha}\}$
\end{itemize}
in the Boolean algebra $\langle a_{b\rest\alpha}:\alpha<l(b)\rangle$.

Given two different branches $b,p$ in $T$, define $b\land p=\min\{\alpha:b(\alpha)\neq p(\alpha)\}$. Following \cite{bellapseudoradial}, if $X$is a $T$-algebra and $b$ is a branch of $T$, we use $\pi_b:X\to X_b$ to denote the natural continuous projection such that $\pi_b(u_b)=u_b$ and $\pi_b(u_q)=z_\alpha$ if $\alpha=b\land q$. We will make use of the following result.

\begin{lemma}\cite{Dowcompactcclosed}\label{lemmaDow}
    Let $X$ be a $\cT$-algebra, $u_b\in X$ and $\{u_n:n\in\w\}\subseteq X\setminus\{u_b\}$. Then $u_n\to u_b$ if and only if $\pi_b(u_n)\to\pi_b(u_b)$ in $X_b$.\qed 
\end{lemma}

    Given a branch $b$ of $T$ and $\gamma\leq l(b)$, let also $X_b^\alpha$ be the truncation of $X_b$ to the $\alpha$-th level. That is, if 
    $$T'=\{t\conc i\in T:t\conc i\subseteq b\land l(t\conc i)<\alpha\}\cup\{t\conc (1-i)\in T:t\conc i\subseteq b\land l(t\conc i)<\alpha\}$$
    then $X_b^\alpha$ is the $T'$-algebra generated by $\{a_t:t\in T'\}$. In particular, $X_b=X_b^{l(b)}$. If $\beta<\alpha$, we will use $z_\beta$ to denote the ultrafilter generated by $\{a_{b\rest\eta}:\eta<\beta\}\cup\{\w\setminus a_{b\rest\beta}\}$ in both $X_b$ and $X_b^\alpha$.\\

We are now ready to proof the main result of the section.

\begin{theorem}\label{sequentiallycompacttalgebra}
    If $T=\{s\in2^{<\p}:l(s)\textnormal{ is a successor ordinal}\}$, and $a_t$ are as described above, then the Stone space of the $T$-algebra generated by $\{a_t:t\in T\}$ is sequentially compact.
\end{theorem}

\begin{proof}
    Let $\cB$ be the Boolean algebra generated by $\{a_t:t\in T\}$ and let us denote by $X$ its Stone space $St(\cB)$. 
    Let $\{u_n:n\in\w\}\subseteq X$ be a one-to-one sequence. 
    We recursively construct a Cantor scheme $\{t_s:s\in2^{<\w}\}\subseteq T$ and $\{Y_s:s\in2^{<\w}\}\subseteq[\w]^\w$ such that for every $s\in2^{<\w}$:
    \begin{enumerate}
        \item $t_{s\conc i}>t_s$ for $i\in2$.
        \item There is $r_s\in T$ such that $t_{s\conc i}=r_s\conc i$ for $i\in2$.
        \item $Y_s=Y_{s\conc 0}\sqcup Y_{s\conc 1}$ (where $\sqcup$ denotes disjoint union).
        \item For $t\in T$, if $t\leq t_s$ then $\{n\in Y_s:u_n\in a_t^*\}$ is cofinite.
    \end{enumerate}

    \vspace{3mm}
     
     Let us start by defining $t_\emptyset=\emptyset$ and $Y_\emptyset=\w$. Now assume we have defined $t_s$ and $Y_s$ satisfying (1)-(4). If $\{u_n:n\in Y_s\}$ is a convergent sequence, we halt the construction since we have already found a convergent subsequence for $\{u_n:n\in\w\}$. 
     So we can assume without loss of generality that $\{u_n:n\in Y_s\}$ is not a convergent sequence. 
     As $X$ is compact, we can find $t'\in T$ such that the clopen set $a_{t'}^*$ splits $\{u_n:n\in Y_s\}$, i.e., both sets $\{n\in Y_s:u_n\in a_{t'}^*\}$ and $\{n\in Y_s:u_n\notin a_{t'}^*\}$ are infinite. 
     We can moreover take $t'$ of minimal possible length. 
     As $t'\in T$ we have that $t'=r\conc j$ for some $r\in T$ and $j\in 2$. Define $t_{s\conc i}=r\conc i$, so that $(2)$ holds.

     From the fact that $\cB$ is a $\cT$-algebra, we get that $a_{t_{s\conc i}}=\w\setminus a_{t_{s\conc (1-i)}}$, hence by defining $Y_{s\conc i}=\{n\in Y_s:u_n\in a_{t_{s\conc i}}^*\}$ for $i<2$ we conclude that $Y_{s\conc i}\in[\w]^\w$ and $Y_s=Y_{s\conc 0}\sqcup Y_{s\conc 1}$, hence $(3)$ holds. 

     That $(1)$ holds follows directly from our inductive hypothesis $(4)$ and $(4)$ holds from the fact that $Y_{s\conc i}\subseteq Y$ and from the minimality of $t'$.

     \vspace{3mm}

    For every $\varphi\in2^\w$ let $Y_\varphi$ be a pseudointersection of $\{Y_{\varphi\rest n}:n\in\w\}$. Since $u_n:n\in\w$ is one-to-one, we can also fix $\varphi\in2^\w$ such that if $q=\sup\{t_{\varphi\rest n}:n\in\w\}$ and $\gamma=l(q)$, then $l(b_{u_n}\land q)<\gamma$ for every $n\in\w$. 
    Let $b\in 2^\p$ be defined by
    $$b(\alpha)=\begin{cases}
    q(\alpha)    & \textnormal{ if }\ \ \alpha<l(q) \\
    1    & \textnormal{ if }\ \ \alpha\geq l(q)
    \end{cases}$$

    Our goal is to show that $\{u_n:n\in Y_\varphi\}$ converges to $u_b$. In order to do this we will use Lemma \ref{lemmaDow}. To simplify our notation let $b_n=b_{u_n}$ and $z_n=\pi_b(u_n)$ for every $n\in Y_\varphi$. We will show by induction that $\{z_n:n\in Y_\varphi\}$ converges to $u_{b\rest\alpha}$ in $X_b^\alpha$ for every $\gamma\leq\alpha\leq\p$.
    
    For $\alpha=\gamma=l(q)$, $(1)$, $(4)$ and the fact that $\{a_\sigma^*:\sigma\in q\}$ is a local base for the ultrafilter $u_q\in X_b^\gamma$ imply that $\{z_n:n\in Y_\varphi\}$ converges to $u_q=u_{b\rest\gamma}$.
    
    A similar argument shows that if this holds for every $\gamma\leq\alpha<\beta$ and $\beta$ is limit, then it also holds for $\beta$, as a typical element in the local base of $u_{b\rest\beta}$ is the intersection of finitely many sets of the form $a_{b\rest \lambda}^*$ with $\lambda<\beta$.

    Assume then that $z_n\to u_{b\rest\alpha}$ for some $\gamma\leq\alpha<\p$. 
    Notice that $a=a_{b\rest\alpha\conc0}$ is a pseudointersection of  $\{a_{b\rest \beta+1}:\beta<\alpha\}$. 
    In particular, for every $n\in Y_\varphi$, if $({t_{\varphi\rest n})\conc i}\subseteq b\rest\alpha$, then $a\subseteq^* a_{({t_{\varphi\rest n}})\conc i}$. On the other hand, $\w\setminus a_{{(t_{\varphi\rest n})}\conc i}=a_{({t_{\varphi\rest n}})\conc (1-i)}\in z_n$, so $a\not\in z_n$, which implies that $a_{{(b\rest\alpha)}\conc1}=\w\setminus a\in z_n$. That is, the only new element in the base filter for $u_{b\rest(\alpha+1)}$ (namely $a_{(b\rest\alpha)\conc1}$) with respect to the family generating $u_{b\rest\alpha}$ belongs to every ultrafilter $z_n$. Hence it is irrelevant for the convergence of the given sequence and therefore $z_n\to p\rest(\alpha+1)$.\\

    By induction we get that $z_n\to u_{b\rest\p}=u_b$ in $X_b^\p=X_b$, and by Lemma \ref{lemmaDow}, we get that $\{u_n:n\in Y_\varphi\}$ converges to $u_b$ in $X$. Therefore $X$ is sequentially compact.
\end{proof}

\begin{cor}
    If $2^\p>\c$, there is a sequentially compact separable space of cardinality $>\c$. In particular, it holds under $\p=\c$\qed
\end{cor}

The same idea as the used in Theorem \ref{sequentiallycompacttalgebra} actually shows a bit more: if we take any acceptable subtree $T$ and repeat the proof, either we find a maximal branch in $T$ given by some $\varphi\in2^\w$, or we can repeat the argument adding 1's until we define a branch (not necessarily of length $\p$). In any case, if $q$ is this branch, the subsequence $\{u_n:n\in Y_\varphi\}$ converges to $u_q$. 

\begin{cor}
    If $T\subseteq2^{<\p}$ is any acceptable tree and the $a_t$ are defined as above, then the $T$-algebra generated by $\{a_t:t\in T\}$ is sequentially compact.\qed
\end{cor}

We now present an alternative way to produce a sequentially compact separable $\cT$-algebra: In the above argument, if there is always a $\varphi\in2^\w$ such that $\{t_{\varphi\rest n}:n\in\w\}$ defines a maximal branch through $T$, then the corresponding subsequence $\{u_n:n\in Y_\varphi\}$ is convergent, and this does not depend on the $a_t$ being chosen as above. So given {\em any} $T$-algebra on an acceptable subtree $T$ of some $2^{<\kappa}$, if $T$ does not include any copy of $2^\w$, there is always one of these branches and the space is sequentially compact. 

Recall that a Kurepa tree is a subtree $K\subseteq2^{<\w_1}$ such that $K$ has height $\w_1$, each level is countable, and $K$ has at least $\w_2$-many branches.
In general, a $\kappa$-Kurepa tree is a tree of height $\kappa$ with levels of size $<\kappa$ and more than $\kappa$ many cofinal branches. 

It is easy to see that given a $\kappa$-Kurepa tree, there is an everywhere $2$-splitting $\kappa$-Kurepa tree and by removing the nodes at limit levels, gives an acceptable $\kappa$-Kurepa tree. 

\begin{cor}\label{Kurepa} Assume the existence of an acceptable 
     ${\mathfrak c}$-Kurepa tree $K\subseteq2^{<{\mathfrak c}}$. Then for any $T$-algebra on $T$, the Stone space is sequentially compact. 
\end{cor}

Note that the existence of a ${\mathfrak c}$-Kurepa tree follows from $\diamondsuit^+({\mathfrak c})$ \cite{Devlin},  so this yields other consistent examples of large separable sequentially compact spaces.






The relevance of this final discussion is that if we were able to build a $T$-algebra over a $\c-$Kurepa tree consistent with $\p<\c$, we would obtain new models where there are sequentially compact separable spaces of size bigger than $\c$.

\section{Questions}

Of course, our main Question \ref{MainQ} has a consistent negative answer in the Miller model, but we do not know whether it is consistent that $X^\kappa$ is countably compact whenever $X$ is a separable sequentially compact space. Even though $MA$ implies the existence of large separable sequentially compact spaces, all the examples presented are compact and so a positive answer to the following is still possible:
\begin{que} Does $MA$ or $PFA$ imply that separable sequentially compact spaces have all powers countably compact?
\end{que}

This would follow with a positive answer to the following question:

\begin{que} Is $(*)$ consistent with all separable sequentially compact spaces have size $\leq \mathfrak c$.  Recall we defined the principle $(*)$: every tree $\pi$-base for $\omega^*$ has height $\mathfrak c$ and a cofinal branch. 
\end{que}
 Of course $(*)$ implies ${\mathfrak h}={\mathfrak b}={\mathfrak c}$ and the last equality implies the existence of $\cT$-algebra on $2^{\mathfrak c}$ (\cite{Sh:984}). And so if there is a ${\mathfrak c}$-Kurepa tree we could restrict this $\cT$-algebra to the Kurepa subtree to obtain a large example. 
 
The $\cT$-algebra construction using a $\c$-Kurepa tree does not require the levels of $T$ to be small, only the weaker consequence that the tree does not embed any copies of $2^\omega$ as a subtree. So let us define a tree $T\subseteq 2^{<\kappa}$, where $\kappa\leq {\mathfrak c}$ to be a {\em $\kappa$-Bernstein-Kurepa tree} if it does not embed $2^\omega$ as a subtree and has more than $\kappa$ many branches.

\begin{que} Is there an $\omega_1$-Bernstein-Kurepa tree that is not a Kurepa-tree? Or is it consistent that there are no $\omega_1$-Kurepa trees but there is an $\omega_1$-Bernstein-Kurepa tree? 
\end{que}

And relevant to the above discussion:

\begin{que} Does ${\mathfrak{b=c}}$ imply the existence of a ${\mathfrak c}$-Bernstein-Kurepa tree?
\end{que}

It is possible that CH implies the existence of an $\omega_1$-Bernstein-Kurepa tree. Or even that there is a ${\mathfrak c}$-Bernstein-Kurepa tree in ZFC. 




\bibliography{seqcompsep.bib}{}

\begin{thebibliography}{10}

\bibitem{BPS}
Bohuslav Balcar, Jan Pelant, and Petr Simon.
\newblock The space of ultrafilters on {${\bf N}$}\ covered by nowhere dense
  sets.
\newblock {\em Fund. Math.}, 110(1):11--24, 1980.

\bibitem{bellapseudoradial}
Angelo Bella, Alan Dow, and Rodrigo Hern\'andez-Guti\'errez.
\newblock Pseudoradial spaces and copies of {$\omega_1+1$}.
\newblock {\em Topology Appl.}, 272:107070, 15, 2020.

\bibitem{BSNCF}
Andreas Blass and Saharon Shelah.
\newblock Near coherence of filters. {III}. {A} simplified consistency proof.
\newblock {\em Notre Dame J. Formal Logic}, 30(4):530--538, 1989.

\bibitem{Devlin}
Keith~J. Devlin.
\newblock {\em Constructibility}.
\newblock Perspectives in Mathematical Logic. Springer-Verlag, Berlin, 1984.

\bibitem{Dowcompactcclosed}
A.~Dow.
\newblock Compact {C}-closed spaces need not be sequential.
\newblock {\em Acta Math. Hungar.}, 153(1):1--15, 2017.

\bibitem{DowJuhasz}
Alan Dow and Istv\'an Juh\'asz.
\newblock On the cardinality of separable pseudoradial spaces.
\newblock {\em Topology Appl.}, 304:Paper No. 107791, 17, 2021.

\bibitem{Sh:984}
Alan~Stewart Dow and Saharon Shelah.
\newblock {An Efimov space from Martin's axiom}.
\newblock {\em Houston J. Math.}, 39(4):1423--1435, 2013.

\bibitem{engelking}
Ryszard Engelking.
\newblock {\em General topology}.
\newblock Monografie Matematyczne, Tom 60. [Mathematical Monographs, Vol. 60].
  PWN---Polish Scientific Publishers, Warsaw, 1977.
\newblock Translated from the Polish by the author.

\bibitem{koppelberg}
Sabine Koppelberg.
\newblock Minimally generated {B}oolean algebras.
\newblock {\em Order}, 5(4):393--406, 1989.

\bibitem{koszmidertalgebras}
Piotr Koszmider.
\newblock Forcing minimal extensions of {B}oolean algebras.
\newblock {\em Trans. Amer. Math. Soc.}, 351(8):3073--3117, 1999.

\bibitem{kunenbook}
Kenneth Kunen.
\newblock {\em Set theory}, volume 102.
\newblock North-Holland Publishing Co., Amsterdam-New York, 1980.
\newblock An introduction to independence proofs.

\bibitem{SS}
C.~T. Scarborough and A.~H. Stone.
\newblock Products of nearly compact spaces.
\newblock {\em Trans. Amer. Math. Soc.}, 124:131--147, 1966.

\bibitem{vanDouwentheintegers}
Eric~K. van Douwen.
\newblock The integers and topology.
\newblock In {\em Handbook of set-theoretic topology}, pages 111--167.
  North-Holland, Amsterdam, 1984.

\bibitem{vaughanSS}
Jerry Vaughan.
\newblock The {S}carborough-{S}tone problem.
\newblock In Elliott Pearl, editor, {\em Open problems in topology. {II}},
  pages xii+763. Elsevier B. V., Amsterdam, 2007.

\end{thebibliography}
\bibliographystyle{plain}
\end{document}